\documentclass[12pt]{article}
\usepackage{font_shortcuts,amsmath,amsthm}
\usepackage{wrapfig,graphicx,caption}

\newtheorem{theorem}{Theorem}[section]

\newtheorem{proposition}[theorem]{Proposition}
\newtheorem{corollary}[theorem]{Corollary}

\theoremstyle{remark}
\newtheorem{remark}[theorem]{Remark}
\numberwithin{equation}{section}
\newcommand{\conj}[1]{\overline{#1}}

\theoremstyle{definition}

\newcommand{\Sp}{\operatorname{Sp}}
\newcommand{\GL}{\operatorname{GL}}
 
\title{The Lie algebra of the fundamental group of a surface as a symplectic module}
\author{Simion Filip}
\date{{\today}}

\begin{document}
%----------------------------
\maketitle

\begin{abstract}
This note provides a formula for the character of the Lie algebra of the fundamental group of a surface, viewed as a module over the symplectic group.
\end{abstract}

\section{Introduction}

To an arbitrary group $G$ one canonically associates a Lie algebra by considering the lower central series and its associated graded. This Lie algebra often carries a lot of information about the group itself and its symmetries. This note provides an explicit description of the symplectic symmetries of the Lie algebra corresponding to the fundamental group of $\Sigma_g$ - a closed surface of genus $g$.

More concretely, denote this group by $G$. Its lower central series is defined recursively by
$$
G^{(i+1)}:= [G^{(i)},G] \hskip 1cm \textrm{ and }\hskip 1cm G^{(1)} := G
$$
In other words, the $i$\textsuperscript{th} group is the group generated by elements which are $(i-1)$-fold commutators.
Next, form the associated graded quotients
\begin{equation}
\frakg_i := G^{(i)}/G^{(i+1)} \hskip 1cm\forall i\geq 1
\label{eqn:def_of_g}
\end{equation}
These quotients are abelian groups and if we consider
$$
\frakg:=\bigoplus_{i\geq 1} \frakg_i
$$
then we obtain a graded Lie algebra, since the bracket $[a,b]:=aba^{-1}b^{-1}$ descends to the level of the quotients.

Since all the constructions were natural, symmetries of $G$ will act on the Lie algebra $\frakg$ (and preserve the grading). Note that conjugating by an element $x\in G$ induces the trivial automorphism on $\frakg$, so we can consider the action of the mapping class group $Mod(\Sigma_g)$ on $\frakg$. Observe further that if $\phi\in Mod(\Sigma_g)$ acts trivially on $\frakg_1$, then $\phi$ acts trivially on all of $\frakg$. 

Indeed, the first assertion is equivalent to $x\phi(x)^{-1}\in G^{(2)}$ for any $x\in G$. Using this and the identity $[ab,c]=a[b,c]a^{-1}[a,c]$ one can check directly that $x\phi(x)^{-1}\in G^{(i+1)}$ for any $x\in G^{(i)}$.

In our present case, since $G$ is the fundamental group of a surface $\Sigma_g$ of genus $g$, it is well-known that $\frakg_1\cong H_1(\Sigma_g,\bZ)$ (see the monograph of Farb and Margalit  \cite[Chapter 5]{FarbMargalit_Primer} for this result and more context). Further, the action of $Mod(\Sigma_g)$ factors through $\Sp_g(\bZ)$, the group of symplectic automorphisms of this $2g$-dimensional space. 

In the case when $G$ is a free group, one obtains the free Lie algebra with an action of $\GL_n$ ($\frakg_1$ is then the standard module).
The representation which occur in this case are classical, see the monograph of Reutenauer \cite{Reutenauer_Free_Lie} for a comprehensive collection of results.
This note describes the representation of $\Sp_g(\bZ)$ on the graded piece $\frakg_N$ as follows.

\begin{theorem}
The character of the representation of $\Sp_g(\bZ)$ on $\frakg_N$, denoted $\chi_N$, is given by the formula
$$
\chi_N = \frac 1N \sum_{d|N}\mu\left(\frac Nd\right) d 
\sum_{k \geq 0} \chi_V^{d-2k,(\frac Nd)}\frac{1}{d-k}\dbinom{d-k}{k}(-1)^k
$$
Here $\mu$ denotes the M\"{o}bius function and $\chi_V$ denotes the character of $\Sp_g$ on $H_1(\Sigma_g)$. The expression $\chi^{x,(y)}$ denotes the function on $\Sp_g$ given by $\chi^{x,(y)}(M)=\chi(M^y)^x$.
\end{theorem}

This formula is proved in Proposition \ref{prop:formulas}, where another equivalent formulation is presented.

\paragraph{Outline of the approach.} The first step is to find the Euler-Poincar\'{e} series of the universal enveloping algebra $\cU(\frakg)$. This is possible due to the algebraic results of Labute, explained in Section 2. The actual computation for $\cU(\frakg)$ is done in Section $3$. The calculation of the character $\chi_N$ of the $N$\textsuperscript{th} graded piece $\frakg_N$ is contained in Section 4.

\paragraph{Acknowledgments.} I would like to thank Prof. Benson Farb for several useful discussions related to this note, as well as for remarks on an earlier version of the text.

\paragraph{Notation and Remarks.} Throughout the text, the following notations and conventions are used:
\begin{itemize}
	\item If $\frakh$ is a Lie algebra, its universal enveloping algebra is denoted $\cU(\frakh)$
	\item Universal enveloping algebras have a natural filtration, and $gr^\bullet \cU(\frakh)$ denotes the associated graded with respect to this filtration
	\item If $M$ is a $\bZ$-module, the free Lie algebra on this module is denoted $\bF(M)$
	\item If the $\bZ$-module $M$ has an extra structure (e.g. grading or $\Sp_g$-module structure), the free Lie algebra $\bF(M)$ inherits it
	\item A closed surface of genus $g$ is fixed and denoted $\Sigma_g$, its homology is denoted $V:=H_1(\Sigma_g,\bZ)$ and its fundamental group is denoted $G$
	\item The Lie algebra $\frakg$ is the one defined by Equation \ref{eqn:def_of_g} and its graded pieces are $\frakg_i$
	\item For any module $M$, its symmetric algebra is denoted
$$
\Sym^\bullet(M) := \bigoplus_{i\geq 0} \Sym^{i}(M)
$$
Any extra structure on $M$ (grading or $\Sp_g$-module) is inherited by $\Sym^\bullet (M)$
	\item In all the situations below, modules will start in homogeneous degree at least $1$ and will be finitely generated in each homogeneous component. These two properties together are stable under taking symmetric algebras or universal enveloping algebras.
	\item A key tool at many steps will be the Poincar\'{e}-Birkhoff-Witt theorem. It states that for any Lie algebra $\frakh$ the natural map
\begin{equation}
gr^\bullet \cU(\frakh) \to \Sym^\bullet(\frakh)
\label{eqn:PBW}
\end{equation}
is an isomorphism. Moreover, this isomorphism is compatible with any extra gradings or $\Sp_g$-module structures on $\frakh$.
\end{itemize}

\section{Preliminary calculations}

This section collects purely algebraic results that will be used when computing the universal enveloping algebra of the Lie algebra $\frakg$.

\paragraph{Setup.} Recall the surjection $\bF_{2g}\onto G$, where $\bF_{2g}$ is the free group on $2g$ generators. To fix notation, the group $G$ is a group freely generated by elements $a_1,b_1,\ldots, a_g,b_g$ with the single relation $[a_1,b_1]\cdots [a_g,b_g]=1$. 

We make use of the results of Labute from \cite{Labute_Group}.
The main theorem of the paper implies that $\frakg$ is a free Lie algebra with a single relation. 
The results can be summarized as follows.
\begin{description}
\item[(i)] Let $V_{2g}:=H_1(\Sigma_g,\bZ)$ be the $\Sp_g$-module which is also canonically identified with the abelianization of $\pi_1$, and let $\bF(V_{2g})$ be the free Lie algebra on this module. The Lie sub-algebra $\frakr\subset \bF(V_{2g})$ is defined by the following canonical short exact sequence.
\begin{equation}
0\to \frakr \into \bF(V_{2g}) \onto \frakg\to 0
\label{eqn:ses_r_F_g}
\end{equation}
\item[(ii)] Consider the element $\rho:=[a_1,b_1]\cdots [a_g,b_g]\in \bF_{2g}$ and let $\overline{\rho}\in \bF(V_{2g})$ be its image in the free Lie algebra. Then $\overline{\rho}\in\frakr$ and moreover, $\frakr$ is the Lie ideal generated by this element.
\end{description}
The following observations will be needed later.
\begin{description}
\item[(i)] The middle and last groups in the sequence \ref{eqn:ses_r_F_g} are naturally $\Sp_g$-modules, and moreover the maps are $\Sp_g$-equivariant. Since the free Lie algebra $\bF(V_{2g})$ can be expressed by tensor operations in terms of the module $V_{2g}$, it is semisimple as an $\Sp_g$-module. 
\item[(ii)] The previous remark implies that $\frakg$ and $\frakr$ are also semisimple $\Sp_g$-modules and their components are derived from the module $V_{2g}$ by tensor operations.
\item[(iii)] For the same reasons as before, the exact sequence \ref{eqn:ses_r_F_g} splits as a sequence of $\Sp_g$-modules.
\end{description}

The last observation yields the following result.

\begin{proposition}
\label{prop:UF=Ug_Ur}
Given the short exact sequence \ref{eqn:ses_r_F_g}, we have an isomorphism of the associated graded of the universal enveloping algebras:
\begin{equation}
gr^\bullet \cU(\bF(V_{2g})) \cong gr^\bullet \cU(\frakg)\otimes_\bZ gr^\bullet \cU(\frakr)
\end{equation}
Moreover, this isomorphism is compatible with the $\Sp_g$-structure on all the terms involved.
\end{proposition}
\begin{proof}
By the Poincar\'{e}-Birkhoff-Witt theorem (see Equation \ref{eqn:PBW}) the claim reduces to the identity
$$
\Sym^\bullet(\bF(V_{2g})) \cong \Sym^\bullet(\frakg)\otimes_\bZ \Sym^\bullet(\frakr)
$$
Because the short exact sequence \ref{eqn:ses_r_F_g} splits as a sequence of $\Sp_g$-modules, the  above isomorphism is implied by the generally valid identity
\begin{equation}
\Sym^\bullet(A\oplus B) = \Sym^\bullet(A)\otimes \Sym^\bullet(B)
\label{eqn:sym_A+B}
\end{equation}
Moreover, this last identity respects any extra gradings or $\Sp_g$-module structures that $A$ or $B$ may have.
\end{proof}

The next task is to understand the Lie algebra $\frakr$, or rather its universal enveloping algebra $\cU(\frakr)$. Consider the adjoint action of $\frakg$ on $\frakr/[\frakr,\frakr]$. It is well-defined and extends to give $\frakr/[\frakr,\frakr]$ the structure of a $\cU(\frakg)$-module.

The structure of $\frakr/[\frakr,\frakr]$ is given by Theorem 1 in \cite{Labute_Lie_Algebra}, which states the following.

\begin{theorem}[\cite{Labute_Lie_Algebra}]
\label{thm:rrr_free_Ug}
As a $\cU(\frakg)$-module, $\frakr/[\frakr,\frakr]$ is free on the generator $\conj{\rho}$. Here, $\rho=[a_1,b_1]\cdots [a_g,b_g]\in \bF_{2g}$ and $\conj{\rho}$ denotes its image in $\frakr$, and by abuse of notation in $\frakr/[\frakr,\frakr]$.
\end{theorem}

To proceed, note that $\frakr$ is a Lie subalgebra of a free Lia algebra, thus it is itself free (see \cite{Shirshov}). Later calculations require a basis on which $\frakr$ is free. This is provided by the next result, which follows from Proposition 2, Section 3 of \cite{Labute_Lie_Algebra} (see also \cite[sec. 4]{Labute_Lie_Algebra}). 

\begin{proposition}
\label{prop:r_to_rrr}
Consider the map $\frakr \to \frakr/[\frakr,\frakr]$, a map of graded Lie algebras and also $\Sp_g$-modules. For each graded component of $\frakr/[\frakr,\frakr]$, take a lift to $\frakr$ of a homogeneous basis (respecting the $\Sp_g$-module structure). Then $\frakr$ is a free Lie algebra on these lifts.
\end{proposition}
\begin{remark}
The results of Labute refer to a situation without an $\Sp_g$ action. However, that result combined with the semisimplicity of the $\Sp_g$-module $\frakr$ and provided the lifts respect the decomposition, yields the statement of the proposition.
\end{remark}

\section{Identifying $\cU(\frakg)$}

The calculations in this section identify the Euler-Poincar\'{e} series of $\cU(\frakg)$ as an $\Sp_g$-module. 

\paragraph{Notation.} For a graded $\Sp_g$-module $M$ with graded pieces $M_i$, i.e. such that $M=\oplus M_i$, define the formal sum
$$
h(M;t):=\sum_i t^i \chi_i
$$
Here $\chi_i:\Sp_g(\bZ)\to \bZ$ is the character of the $\Sp_g$-representation on $M_i$. The generating function $h(M;t)$ is viewed as a formal combination of functions on the group. Most often, the formal variable $t$ will be omitted from the expression for $M$

\begin{proposition}
\label{prop:basic_facts_EP}
\begin{description}
\item[(i)] If $A$, $B$ and $C$ are graded $\Sp_g$-modules and $C\cong A\otimes B$ (as graded $\Sp_g$-modules) then
$$
h(C;t) = h(A;t)\cdot h(B;t)
$$
\item[(ii)] Suppose $\frakh$ is a free Lie algebra on a $\Sp_g$-module $M$, with elements of $M$ in homogeneous degree $2$. 
If $\cU(\frakh)$ is its enveloping algebra (with the induced $\Sp_g$-action) then
$$
h(\cU(\frakh);t) = \frac{1}{1 - h(M;t)}
$$
\end{description}
\end{proposition}
\begin{remark}
With the above notation, one has $h(M;t)=t^2\cdot\chi_M$, where $\chi_M$ is the character of the $\Sp_g$-module $M$.
\end{remark}
\begin{proof}
The first part follows because characters multiply upon taking tensor products. For the second part, observe that $\cU(\frakh)$ is the free tensor algebra on $M$, therefore 
$$
h(\cU(\frakh);t) = 1 + h(M;t) + h(M;t)^2+\cdots
$$
This is exactly the statement in the second part.
\end{proof}

\begin{corollary}
In the notation of Section 2, The Euler-Poincar\'{e} series of $\cU(\frakr)$ is
\begin{equation}
h(\cU(\frakr)) = \frac{1}{1- h(\frakr/[\frakr,\frakr])}
\label{eqn:hU(r)=hr}
\end{equation}
\end{corollary}
\begin{proof}
This follows from the second part of the previous proposition, combined with Proposition \ref{prop:r_to_rrr}.
\end{proof}

\begin{proposition}
With the notation from Section 2, the following formula holds
\begin{equation}
h(\cU(\frakg);t) = \frac{1}{1-t\cdot\chi_V + t^2}
\label{eqn:hUg}
\end{equation}
The character $\chi_V$ is that of the standard representation of $\Sp_g(\bZ)$ on $V$ - a $2g$-dimensional symplectic $\bZ$-module.
\end{proposition}

\begin{proof}
Proposition \ref{prop:UF=Ug_Ur} and the first part of Proposition \ref{prop:basic_facts_EP} imply the identity
\begin{equation}
h(\cU(\bF(V_{2g}))) = h(\cU(\frakr))\cdot h(\cU(\frakg))
\end{equation}

The second part of Proposition \ref{prop:basic_facts_EP} implies that
\begin{equation}
h(\cU(\bF(V_{2g}))) = \frac{1}{1-t\cdot\chi_V}
\label{eqn:hUv=hv}
\end{equation}

Because $\frakr/[\frakr,\frakr]$ is a free $\cU(\frakg)$-module (Theorem \ref{thm:rrr_free_Ug}), with generator in homogeneous degree two, we have
\begin{equation}
\label{eqn:hrrr=hUh}
h(\frakr/[\frakr,\frakr];t) = t^2 h(\cU(\frakg))
\end{equation}
Combining equations \ref{eqn:hU(r)=hr}, \ref{eqn:hUv=hv} and \ref{eqn:hrrr=hUh}, we find the identity
$$
\frac{1}{1-t\cdot \chi_V} =  \frac{h(\cU(\frakg))}{1-t^2h(\cU(\frakg))}
$$
Rearranging, we find
$$h(\cU(\frakg);t) = \frac{1}{1-t\cdot\chi_V + t^2}$$
This is exactly the desired expression.
\end{proof}

\section{Computing the characters of the representation}

\paragraph{Outline.} Equation \ref{eqn:hUg} gives an expression for the character of the universal enveloping algebra. This section contains the calculation of the character for the Lie algebra $\frakg$ itself. The approach requires determining the character of the representation $\frakg_i$ from knowing the character of $\Sym^{\bullet}\frakg_i$. This, in turn, is based on the M\"{o}bius inversion formula.

\paragraph{The First reduction.} 

At the level of Euler-Poincar\'{e} series, $\cU(\frakg)$ and $gr^\bullet\cU(\frakg)$ coincide because the filtration is compatible with the $\Sp_g$-module structure. Using the Poincar\'{e}-Birkhoff-Witt theorem, we find
$$
gr^\bullet \cU(\frakg) = \Sym^\bullet(\frakg) = \bigotimes_{i\geq 1} \Sym^\bullet(\frakg_i)
$$
The second equality above follows from formula \ref{eqn:sym_A+B}. Reading this equality at the level of Euler-Poincar\'{e} series, combined with equation \ref{eqn:hUg} and part (i) of Proposition \ref{prop:basic_facts_EP}, we find
$$
\frac{1}{1-t\cdot \chi_V + t^2} = \prod_{i\geq 1} h(\Sym^\bullet(\frakg_i);t)
$$
Note that $\frakg_i$ sits in homogeneous degree $i$, so the above product makes sense as a formal power series in $t$, and moreover both sides are of the form $$1+(\ldots)t+(\ldots)t^2+\cdots$$

This means we can take formally a $\log$ of both power series to find
\begin{equation}
\label{eqn:log_hUg}
-\log (1-t\cdot \chi_V + t^2) = \sum_{i\geq 1} \log h(\Sym^\bullet(\frakg_i);t)
\end{equation}

The following result relates the character of a module and its symmetric powers.

\begin{proposition}
Let $\GL_k(\bZ)$ act naturally on the free $\bZ$-module $\bZ^k=:M$. Assign homogeneous degree $i$ to the module $M$ and form the graded module $\Sym^\bullet (M)$. Then the following equality of formal power series in $t$ holds:
$$
\log h(\Sym^\bullet(M);t) = \sum_{d\geq 1} \frac{\chi^{(d)}}{d}\cdot t^{i\cdot d}
$$
The function $\chi^{(d)}:\GL_k(\bZ)\to \bZ$ is defined by $\chi^{(d)}(A)=\chi(A^d)$.
\end{proposition}
\begin{proof}
The statement of the proposition is that of an equality between coefficients of two formal power series. These coefficients are polynomials with rational coefficients in the entries of matrices in $\GL_k(\bZ)$. To check such an equality, it suffices to check it for the group $\GL_k(\bC)$, and moreover the dependence in $t$ can be viewed as an analytic function in a small disk around the origin. To check this last identity, it suffices to check an equality between two analytic function on $\bD\times GL_k(\bC)$ (where $\bD$ is a small disk). It now suffices to check this for diagonalizable matrices in $\GL_k(\bC)$, since this set is dense.

Take one such diagonalizable $A$ and suppose its eigenvalues on $M\otimes_\bZ \bC$ are $\lambda_1,\ldots,\lambda_k$. Then the sum of the eigenvalues on $\Sym^\bullet(M)$ are, after arranging them by homogeneity using the variable $t$, given by
$$
\prod_{j=1}^k \frac{1}{1-t^i\lambda_j}
$$
Taking logs of the above, we see that the function $A\mapsto \log h(\Sym^\bullet(M);t)(A)$ is given by
$$
A\mapsto \sum_{n\geq 1} \frac{\lambda_1^d+\ldots+\lambda_k^d}{d} \cdot t^{i\cdot d}
$$ 
On the right-hand side, the operator $A$ is first raised to the $d$\textsuperscript{th} power and then trace is taken. So we find what we wanted:
$$
\log h(\Sym^\bullet(M);t) = \sum_{d\geq 1} \frac{\chi^{(d)}}{d}\cdot t^{i\cdot d}
$$
\end{proof}
Using the previous proposition, we can rewrite right-hand side of Equation \ref{eqn:log_hUg} as follows
\begin{align}
\sum_{i\geq 1} \log h(\Sym^\bullet(\frakg_i);t) &=
\sum_{i\geq 1}\sum_{d\geq 1} \frac{\chi_i^{(d)}}{d}\cdot t^{i\cdot d}\notag \\
&= \sum_{N\geq 1} t^N \sum_{d\cdot i=N} \frac{\chi_i^{(d)}}{d}
\label{eqn:RHS}
\end{align}
In this formula $\chi_i$ denotes the character of the $i$\textsuperscript{th} graded piece $\frakg_i$.

\paragraph{The Second Reduction.}

There are two relatively convenient ways to express the left hand side of Equation \ref{eqn:log_hUg} and we do not want to pick out one yet. Let us therefore introduce the notation
\begin{equation}
\label{eqn:def_A_N}
-\log (1-t\chi_V + t^2) = \sum_{N\geq 1} A_N t^N
\end{equation}
With this notation, the next proposition gives an expression for the character~$\chi_N$.

\begin{proposition}
\label{prop:chi_nonexplicit}
For the $N$\textsuperscript{th} graded piece $\frakg_N$, the formula for its character is
$$
N\cdot \chi_N = \sum_{d|N} \mu\left(\frac{N}{d}\right) \cdot d \cdot A_d^{(\frac Nd)}
$$
Here, $\mu$ is the M\"{o}bius function and $A_d^{(\frac Nd)}$ represents the function $A_d$, precomposed with the operation of taking the $\frac Nd$ power of a matrix.
\end{proposition}
\begin{proof}
The expression for the right-hand side of Equation \ref{eqn:log_hUg} provided by Equation \ref{eqn:RHS}, combined with the definition of $A_n$ yields
$$
A_n = \sum_{d|n} \frac{1}{d} \cdot \chi_{\frac nd}^{(d)}
$$
Multiplying by $n$, this can be rewritten as
\begin{equation}
\label{eqn:Dir_convolution}
n\cdot A_n = \sum_{d|n} \frac nd \cdot \chi_{\frac nd}^{(d)}
\end{equation}
To express a given character $\chi_N$ using the above relation, take in the above some $n|N$ and precompose the above identity with taking $\frac Nn$ powers and multiply by $\mu(\frac Nn)$ to find
$$
n \cdot A_n^{(\frac Nn)}\cdot \mu\left(\frac Nn\right) = \mu\left(\frac Nn\right)\cdot \sum_{d|n} \frac nd \cdot \chi_{\frac nd}^{(d\cdot n)}
$$
Keep $N$ fixed and sum the above equation over all $n|N$ to find
\begin{align*}
\sum_{n|N} n A_n^{(\frac Nn)}\mu\left(\frac Nn\right) &=
 \sum_{n|N}\mu\left(\frac Nn\right)\sum_{d|n}\frac nd\cdot \chi_{\frac nd}^{(d\cdot \frac Nn)}=\\
&= \sum_{x|N}\,\, \sum_{d|\frac Nx} \mu\left(\frac N{xd}\right) x \chi_x^{(\frac Nx)}=\\
&= \mu(1)N \chi_N^{(1)} = N\chi_N
\end{align*}
In passing to the second line above, we changed summation to the variable $x:=\frac nd$. In passing to the last line, we used the standard fact about the M\"{o}bius function that $\sum_{d|B}\mu(d)=0$, unless $B=1$. This finishes the proof of the proposition.
\end{proof}
\begin{remark} Recall that for two sequences $A_n, B_n$, their Dirichlet convolution $C_n = (A*B)_n$ is defined by
$$C_n:=\sum_{d|n}A_d\cdot B_{\frac nd}$$
With this in mind, the computation above can be interpreted as follows. Equation \ref{eqn:Dir_convolution} reads expresses $nA_n$ as a Dirichlet product
$$nA_n = n\chi_n * [\bullet]^{(n)}$$
where $[\bullet]^{(n)}$ is the operator taking a matrix to its $n$\textsuperscript{th} power. Since we are interested in $n\chi_n$ we need the Dirichlet inverse of $[\bullet]^{{(n)}}$ and this is given by  $\mu(n)[\bullet]^{(n)}$.
\end{remark}

\paragraph{The formulas.} The proposition below provides explicit formulas for both $A_N$ and $\chi_N$.
\begin{proposition}
\label{prop:formulas}
Two (equivalent) expressions for $A_N$ are
\begin{equation}
\label{eqn:A_N_explicit}
A_N = \sum_{k \geq 0} \chi_V^{N-2k}\frac{1}{N-k}\dbinom{N-k}{k}(-1)^k
\end{equation}
\begin{equation}
\label{eqn:A_N_two_terms}
A_N = \frac 1N \left[ \left(\frac{\chi_V + \sqrt{\chi_V^2-4}}{2}\right)^N  + \left(\frac{\chi_V - \sqrt{\chi_V^2-4}}{2}\right)^N\right]
\end{equation}
Combined with the formula given in Proposition \ref{prop:chi_nonexplicit}, we get two (equivalent) formulas for the characters
\begin{equation}
\chi_N = \frac 1N \sum_{d|N}\mu\left(\frac Nd\right) d 
\sum_{k \geq 0} \chi_V^{d-2k,(\frac Nd)}\frac{1}{d-k}\dbinom{d-k}{k}(-1)^k
\end{equation}
\begin{equation}
\chi_N = \frac 1N \sum_{d|N}\mu\left(\frac Nd\right) 
\left[ \left(\frac{\chi_V^{(\frac Nd)} + \sqrt{\chi_V^{2,(\frac Nd)}-4}}{2}\right)^d  + \left(\frac{\chi_V^{(\frac Nd)} - \sqrt{\chi_V^{2,(\frac Nd)}-4}}{2}\right)^d\right]
\end{equation}
In the last two equations, for a character $\chi$, the expression $\chi^{x,(y)}$ is defined by $\chi^{x,(y)}(M)=\chi(M^y)^x$.
\end{proposition}
\begin{proof}
Only formulas \ref{eqn:A_N_explicit} and \ref{eqn:A_N_two_terms} require proof, since they imply the other two directly. Recall the definition of $A_N$ given by Equation \ref{eqn:def_A_N}
$$
\sum_{N\geq 1} A_N t^N = -\log (1-t\chi_V + t^2)
$$
The first formula for $A_N$ follows from the following rewriting
$$
-\log (1-t\chi_V + t^2) = -\log \left(1-t(\chi_V -t)\right) = \sum_{N\geq 1} \frac 1N t^N\left(\chi_V - t\right)^N
$$
Expanding each individual term using the binomial formula, we arrive at the formula given by equation \ref{eqn:A_N_explicit}.

For the second formula, we can factor the quadratic $1-t\chi_V + t^2$ to find
$$
-\log(1-t\chi_V + t^2) = -\log\left( \left(1-t\cdot\frac{\chi_V-\sqrt{\chi_V^2-4}}{2}\right)
\cdot \left(1-t\cdot \frac{\chi_V+\sqrt{\chi_V^2-4}}{2}\right) \right)
$$
So we have
\begin{align*}
\sum_{N\geq 1} A_N t^N &= -\log \left(1-t\cdot \frac{\chi_V-\sqrt{\chi_V^2-4}}{2}\right) - \log \left(1-t\cdot\frac{\chi_V+\sqrt{\chi_V^2-4}}{2}\right) =\\
&= \sum_{N\geq 1} \frac {t^N}N \left[ \left(\frac{\chi_V + \sqrt{\chi_V^2-4}}{2}\right)^N  + \left(\frac{\chi_V - \sqrt{\chi_V^2-4}}{2}\right)^N\right]
\end{align*}
This is exactly the second formula for $A_N$.
\end{proof}

\bibliography{bibliography_Lie(pi_1)}
\bibliographystyle{alpha}
\end{document}